\documentclass[12pt, reqno]{amsart}

\author[S.~Cerreia-Vioglio]{Simone Cerreia-Vioglio}
\address{Universit\'a ``Luigi Bocconi''\\Department of Decision Sciences\\Milan, Italy}
\email{simone.cerreia@unibocconi.it}

\author[P.~Leonetti]{Paolo Leonetti}
\address{Universit\'a degli Studi dell'Insubria, Department of Economics, via Monte Generoso 71, 21100 Varese, Italy}
\email{leonetti.paolo@gmail.com}
\author[F.~Maccheroni]{Fabio Maccheroni}
\address{Universit\'a ``Luigi Bocconi''\\Department of Decision Sciences\\Milan, Italy}
\email{fabio.maccheroni@unibocconi.it}

\author[M.~Marinacci]{Massimo Marinacci}
\address{Universit\'a ``Luigi Bocconi''\\Department of Decision Sciences\\Milan, Italy}
\email{massimo.marinacci@unibocconi.it}

\title{Capacities and Choquet Averages of Ultrafilters}

\usepackage{amsmath}
\usepackage{amssymb}
\usepackage{amsthm}
\usepackage[left=3.5cm, right=3.5cm, paperheight=11.8in]{geometry}
\usepackage{hyperref}
\usepackage{fancyhdr}
\usepackage{enumitem}
\usepackage{bm}
\usepackage{comment}
\usepackage{nicefrac}
\usepackage{mathrsfs}
\usepackage{graphicx}
\usepackage[utf8]{inputenc}

\AtBeginDocument{%
   \def\MR#1{}
}

\newtheorem{thm}{Theorem}[section]
\newtheorem{cor}[thm]{Corollary}%[section]
\newtheorem{lem}[thm]{Lemma}
\newtheorem{prop}[thm]{Proposition}

\theoremstyle{definition} 
\newtheorem{defi}[thm]{Definition}%[section]
\let\olddefi\defi
\renewcommand{\defi}{\olddefi\normalfont}
\newtheorem{example}[thm]{Example}
\let\oldexample\example
\renewcommand{\example}{\oldexample\normalfont}
\newtheorem{rmk}[thm]{Remark}
\let\oldrmk\rmk
\renewcommand{\rmk}{\oldrmk\normalfont}

\theoremstyle{remark}
\newtheorem{claim}{\textsc{Claim}}
\newtheorem*{claim*}{\textsc{Claim}}

{\textsc{Cerreia-Vioglio}, \textsc{Leonetti}, \textsc{Maccheroni} and \textsc{Marinacci}}
\hypersetup{
    pdftitle={Choquet Averages of Ultrafilters},
    pdfauthor={},
    pdfmenubar=false,
    pdffitwindow=true,
    pdfstartview=FitH,
    colorlinks=true,
    linkcolor=blue,
    citecolor=green,
    urlcolor=cyan
}

\providecommand{\MR}[1]{}

\providecommand{\MR}{\relax\ifhmode\unskip\space\fi MR }
\providecommand{\href}[2]{#2}
                                 
\begin{document}

\maketitle
\thispagestyle{empty}

\begin{abstract}
\noindent %Given an ideal $\mathcal{I}$ on $\mathbf{N}$, w
We show that a normalized capacity $\nu: \mathcal{P}(\mathbf{N})\to \mathbf{R}$ is invariant with respect to an ideal $\mathcal{I}$ on $\mathbf{N}$ if and only if it can be represented as a Choquet average of $\{0,1\}$-valued finitely additive probability measures corresponding to the ultrafilters containing the dual filter of $\mathcal{I}$. 
This is obtained as a consequence of an abstract analogue in the context of Archimedean Riesz spaces. 
\end{abstract}

\section{Introduction}\label{sec:intro}

In this paper, we study normalized capacities which could be defined on the quotient $\mathcal{P}(\mathbf{N})/\mathcal{I}$, where $\mathcal{I}$ is a given ideal on the positive integers $\mathbf{N}$. 
Our goal is to represent them as suitable averages of $\{0,1\}$-valued finitely additive probability measures corresponding to the ultrafilters containing the dual filter of $\mathcal{I}$. 
In what follows, we define the notions of normalized capacity, ideal, Choquet integral, and we introduce the topological requirements needed for our results. This will allow us to discuss them formally and prove them. 

%\textcolor{red}{More precisely}, 

\subsection*{Normalized capacities and ideals}  Given a measurable space $(S,\Sigma)$, where $\Sigma$ is a $\sigma$-algebra of subsets of $S$, a \emph{normalized capacity} $\nu: \Sigma \to \mathbf{R}$ is a monotone set function (i.e., $\nu(A)\le \nu(B)$ for all $A,B \in \Sigma$ with $A\subseteq B$) such that $\nu(\emptyset)=0$ and $\nu(S)=1$. 
A family $\mathcal{I}\subseteq \mathcal{P}(\mathbf{N})$ is an \emph{ideal} if it is closed under subsets,  finite unions, and $\mathbf{N}\notin \mathcal{I}$.  
Unless otherwise stated, we assume that 
%$\mathcal{I}$ is not the whole power set and that it 
$\mathcal{I}$ 
contains the family $\mathrm{Fin}$ of finite sets. 
Let 
$\mathcal{I}^\star:=\{A\subseteq \mathbf{N}: A^c \in \mathcal{I}\}$ be its dual filter. Note that $\mathcal{I}$ is maximal with respect to inclusion if and only if $\mathcal{I}^\star$ is a free ultrafilter on $\mathbf{N}$. 
Two important examples of ideals are: (i) the family of asymptotic density zero sets 
$$
\mathcal{Z}:=\{A\subseteq \mathbf{N}: \mathsf{d}^\star(A)=0\}, 
$$ 
where $\mathsf{d}^\star$ is the upper asymptotic density defined by 
$
%\forall A\subseteq \mathbf{N}, \quad 
\mathsf{d}^\star(A)=\limsup_{n}|A\cap [1,n]|/n
$ 
for all $A\subseteq \mathbf{N}$; (ii) the summable ideal 
$$
\mathcal{I}_{1/n}:=\left\{A\subseteq \mathbf{N}: \sum\nolimits_{a \in A}1/a<\infty\right\}.
$$
We refer to \cite{MR2777744} for a survey on ideals and associated filters. 
Finally, a normalized capacity $\nu: \mathcal{P}(\mathbf{N}) \to \mathbf{R}$ is said to be $\mathcal{I}$\emph{-invariant} provided that $\nu(A)=\nu(B)$ whenever the symmetric difference $A\bigtriangleup B$ belongs to the ideal $\mathcal{I}$.

%Given a measurable space $(S,\Sigma)$, where $\Sigma$ is a $\sigma$-algebra of subsets of $S$, let $\mathrm{ba}(S,\Sigma)$ be the vector space of bounded, signed, and finitely additive set functions $\mu: \Sigma \to \mathbf{R}$. We endow $\mathrm{ba}(S,\Sigma)$, as well as all its subsets, with the weak$^\star$ topology. 

%\textcolor{red}{[Complete]}
%Also

\subsection*{Choquet integrals} 
The idea of normalized capacity naturally calls for the notion of (nonlinear) integral. 
%To formalize the above statement about averages, we need the notion of Choquet integral. 
Consider a bounded $\Sigma$-measurable function  $x: S\to \mathbf{R}$ and a normalized capacity $\nu: \Sigma \to \mathbf{R}$. The \emph{Choquet integral} of $x$ with respect to $\nu$ is the quantity 
$$%\begin{equation}\label{eq:Choquetintegral}
%\int_S f\,\mathrm{d}\nu:= \int_0^\infty \nu(\{s \in S: f(s)\ge t\})\,\mathrm{d}t+\int_{-\infty}^0 \left(\nu(\{s \in S: f(s)\ge t\})-\nu(S)\right)\,\mathrm{d}t,
\int_S x\,\mathrm{d}\nu:= \int_0^\infty \nu(x\ge t)\,\mathrm{d}t+\int_{-\infty}^0 \left[\nu(x\ge t)-\nu(S)\right]\,\mathrm{d}t,
$$%\end{equation}
where the integrals on the right hand side are meant to be improper Riemann integrals. 
This naturally generalizes the standard notion of integral since the two coincide when $\nu$ is finitely additive. From a functional point of view, the Choquet integral fails to be linear, but it is characterized by comonotonic additivity (see the seminal paper of Schmeidler \cite{MR835875}). % cf. also Zhou \cite{MR1373649}. 
In terms of definitions, note that, to compute the Choquet integral, one only needs a monotone set function defined over the upper level sets of the integrand and not necessarily over the entire $\sigma$-algebra $\Sigma$ (this is something we will exploit in our proofs). 
%We refer the reader to \cite{marinacci2004} for a detailed introduction to the theory of Choquet integration. 
Nonadditive set functions and their integrals are widely used in applications: for applications in economics, see \cite{marinacci2004}; for applications in probability and statistics, see \cite{MR3503706, MR2135316, MR1145491}.

\subsection*{Topology} 
Given the set $\mathbf{N}$ of positive integers, 
we consider the Stone--\u{C}ech compactification $\beta\mathbf{N}$. 
Recall that $\beta\mathbf{N}$ is homeomorphic to the space of ultrafilters $\mathcal{F}$ on $\mathbf{N}$, which we still denote by $\beta\mathbf{N}$ and is topologised by the base of clopen subsets $\{\{\mathcal{F} \in \beta\mathbf{N}: A \in \mathcal{F}\}: A\subseteq \mathbf{N}\}$. 
By $\mathrm{Ult}(\mathcal{I})$ we denote the compact subspace of free ultrafilters  which contain the dual filter of a given ideal $\mathcal{I}$, that is, 
$$
\mathrm{Ult}(\mathcal{I}):=\{\mathcal{F} \in \beta\mathbf{N}: \mathcal{I}^\star \subseteq \mathcal{F}\}.
$$
We endow $\mathrm{Ult}(\mathcal{I})$ with its relative topology.  
Note that the subspace $\mathrm{Ult}(\mathcal{I})$, sometimes called \textquotedblleft support set,\textquotedblright
\hspace{.5mm}has been introduced in Henriksen \cite{MR108720} and further studied in \cite{MR271752, MR1095221, 
%MR1734462, 
MR1372186, MR279609} in the context of ideals generated by nonnegative regular summability matrices,\footnote{However, it is known \cite[Proposition 13]{MR3405547} that, if we regard ideals as subsets of the Cantor space $\{0,1\}^{\mathbf{N}}$ endowed with the product topology, then the   latter sets are necessarily $F_{\sigma\delta}$ (hence, the notion of support set was applied only to a restricted class of ideals).} 
cf. also \cite{MR3405547, MR3883309}. 
Given an ultrafilter $\mathcal{F} \in \mathrm{Ult}(\mathcal{I})$, there is a corresponding $\{0,1\}$-valued finitely additive probability measure $\mu_{\mathcal{F}}$ defined by 
$$
\mu_{\mathcal{F}}(A)=1
\,\, \text{ if and only if }\,\,
A \in \mathcal{F}.
$$ 

%More precisely, we will show that a normalized capacity which is $\mathcal{I}$-invariant will be an \textquotedblleft average\textquotedblright \hspace{.5mm}of $\{0,1\}$-valued probability charges $\mu_{\mathcal{F}}$ (i.e., finitely additive probability measures) corresponding to the ultrafilters $\mathcal{F} \in \mathrm{Ult}(\mathcal{I})$, where $\mu_\mathcal{F}$ is defined by
%$$
%\mu_{\mathcal{F}}(A)=1
%\,\, \text{ if and only if }\,\,
%A \in \mathcal{F}.
%$$ 
Finally, given a topological space $X$, denote by $\mathscr{B}(X)$ its Borel $\sigma$-algebra (recall that if $Y\subseteq X$ is endowed with its relative topology then $\mathscr{B}(Y)=\{A\cap Y: A \in \mathscr{B}(X)\}$).

\medskip

\subsection*{Results}  We are now ready to state our main result. 

\begin{thm}
\label{thm:mainN} 
Let $\nu :\mathcal{P}(\mathbf{N})\rightarrow \mathbf{R}$
be a normalized capacity and $\mathcal{I}$ an ideal on $\mathbf{N}$. The
following statements are equivalent:
\begin{enumerate}[label={\rm (\roman{*})}]
\item \label{item:a1} $\nu$ is $\mathcal{I}$-invariant\textup{;}
\item \label{item:a2} There exists a normalized capacity $\rho: \mathscr{B}(\mathrm{Ult}(\mathcal{I}))\to \mathbf{R}$ such that 
$$
\forall x \in \ell_\infty, \quad 
\int_{\mathbf{N}} x \, \mathrm{d}\nu= \int_{\mathrm{Ult}(\mathcal{I})}\, \left(\int_{\mathbf{N}} x \, \mathrm{d}\mu_{\mathcal{F}}\right)\, \mathrm{d}\rho(\mathcal{F})\textup{;}
$$
\item  \label{item:a3} There exists a normalized capacity $\rho: \mathscr{B}(\mathrm{Ult}(\mathcal{I}))\to \mathbf{R}$ such that 
$$
\forall A\subseteq \mathbf{N}, \quad 
\nu(A)= \int_{\mathrm{Ult}(\mathcal{I})}\, \mu_{\mathcal{F}}(A)\, \mathrm{d}\rho(\mathcal{F})\textup{.}
$$
\end{enumerate}
\end{thm}

It is easy to see that if $\nu :\mathcal{P}(\mathbf{N})\rightarrow \mathbf{R}
$ is an \emph{abstract upper density} \cite{MR3863054, FILIPOW2019}, that
is, a normalized capacity which is also diffuse (i.e., $\nu (A)=0$ for all $%
A\in \mathrm{Fin}$) and subadditive (namely, $\nu (A\cup B)\leq \nu (A)+\nu
(B)$ for all $A,B\subseteq \mathbf{N}$), then%
\begin{equation}
\mathcal{Z}_{\nu }:=\{A\subseteq \mathbf{N}:\nu (A)=0\}  \label{eq:Znu}
\end{equation}%
is an ideal which contains $\mathrm{Fin}$ (in particular, $\mathcal{Z}=%
\mathcal{Z}_{\mathsf{d}^{\star }}$) and $\nu $ is $\mathcal{Z}_{\nu }$%
-invariant. This gives us the following easy corollary:

\begin{cor}
\label{cor:simple} If\ $\nu :\mathcal{P}(\mathbf{N})\rightarrow \mathbf{R}$
is an abstract upper density, then there exists a normalized capacity $\rho :%
\mathscr{B}(\mathrm{Ult}(\mathcal{Z}_{\nu }))\rightarrow \mathbf{R}$ such
that 
\begin{equation*}
\forall A\subseteq \mathbf{N},\quad \nu (A)=\int_{\mathrm{Ult}(\mathcal{Z}%
_{\nu })}\,\mu _{\mathcal{F}}(A)\,\mathrm{d}\rho (\mathcal{F}),
\end{equation*}%
where $\mathcal{Z}_{\nu }$ is the ideal defined in \eqref{eq:Znu}.
\end{cor}

Abstract upper densities (and, more generally, submeasures) have been
commonly used in number theory, functional analysis, and economics. Examples
include the upper asymptotic, upper analytic, upper P\'{o}lya, and upper
Banach densities \cite{MR4054777}, exhaustive submeasures 
%\cite{MR2456888}, %\cite{MR701524, MR2456888}, 
\cite{MR701524, MR2456888}, nonpathological and lower semicontinuous
submeasures %\cite{MR1711328}, 
\cite{MR1711328, MR1708146}, duals of exact games %\cite{marinacci2004}, 
\cite{marinacci2004, MR381720}, and, of course, finitely additive and $%
\sigma $-additive measures.

A remarkable example is suggested by the well-known characterization of
analytic P-ideals $\mathcal{I}$ (i.e., analytic ideals such that if $(A_n)$
is a sequence in $\mathcal{I}$ then there exists $A \in \mathcal{I}$ such
that $A_n\setminus A \in \mathrm{Fin}$ for all $n \in \mathbf{N}$). Indeed,
a result due to Solecki states that an ideal $\mathcal{I}$ is an analytic
P-ideal if and only if there exists a lower semicontinuous submeasure $%
\varphi: \mathcal{P}(\mathbf{N}) \to [0,\infty]$ such that 
\begin{equation}  \label{eq:analyticP}
\mathcal{I}=\mathrm{Exh}(\varphi):=\{A\subseteq \mathbf{N}: \|A\|_\varphi=0\}
\end{equation}
with $\|\mathbf{N}\|_\varphi<\infty$, where $\|A\|_\varphi:=\lim_n
\varphi(A\setminus [1,n])$ represents the \textquotedblleft mass at
infinity\textquotedblright\, of each $A\subseteq \mathbf{N}$ (see, e.g., \cite[Theorem 1.2.5(b)]{MR1711328}). 
Replacing $\|\cdot\|_\varphi$ with $%
\|\cdot\|_\varphi/\|\mathbf{N}\|_\varphi$ if necessary, we can assume
without loss of generality that $\|\mathbf{N}\|_\varphi=1$. Now, it is easy 
to see that $\|\cdot\|_\varphi$ is an abstract upper density and, in
particular, it is $\mathcal{I}$-invariant. % by \cite[Lemma 1.3.3(b)]{MR1711328}.
Therefore, thanks to Corollary \ref{cor:simple}, we obtain:

\begin{cor}
\label{cor:analyticP} If $\mathcal{I}=\mathrm{Exh}(\varphi)$ is an analytic P-ideal as in \eqref{eq:analyticP}, then there exists a normalized capacity $%
\rho: \mathscr{B}(\mathrm{Ult}(\mathcal{I}))\to \mathbf{R}$ such that 
\begin{equation*}
\forall A\subseteq \mathbf{N}, \quad \|A\|_\varphi= \int_{\mathrm{Ult}(%
\mathcal{I})}\, \mu_{\mathcal{F}}(A)\, \mathrm{d}\rho(\mathcal{F}).
\end{equation*}
\end{cor}

On the other hand, it is also clear that every ideal $\mathcal{I}$ coincides
with $\mathcal{Z}_\nu$, where $\nu: \mathcal{P}(\mathbf{N}) \to \mathbf{R}$
is the $\{0,1\}$-valued abstract upper density such that $\nu(A)=1$ if and
only if $A\notin \mathcal{I}$. Lastly, note that Theorem \ref{thm:mainN}
holds also for normalized capacities $\nu$ which are not subadditive, e.g.,
the lower asymptotic density $\mathsf{d}_\star: \mathcal{P}(\mathbf{N}) \to 
\mathbf{R}$ defined by $\mathsf{d}_\star(A):=\liminf_n |A\cap [1,n]|/n$ for
all $A\subseteq \mathbf{N}$: in this case, the set $\{A\subseteq \mathbf{N}: 
\mathsf{d}_\star(A)=0\}$ is, of course, not closed under finite unions, however
the normalized capacity $\mathsf{d}_\star$ is $\mathcal{Z}$-invariant.

The proof of Theorem \ref{thm:mainN} is given in Section \ref{sec:mainproof}%
. An abstract version in the context of Riesz spaces will be given in
Theorem \ref{thm:Archimedean}.

%%%%%%%%%%%%%%%%%%%%%%%%%%%%%

\section{An abstract version on Riesz spaces}

Recall that an \emph{ordered vector space} is a real vector space $X$ with a
compatible partial order $\leq $, that is, $\alpha x\leq \alpha y$ and $%
x+z\leq y+z$ for all scalars $\alpha \geq 0$ and all vectors $x,y,z\in X$
with $x\leq y$. An ordered vector space $X$ is a \emph{Riesz space }when $X$
is also a lattice. In addition, $X$ is said to be \emph{Archimedean} if $%
0\leq nx\leq y$ for some $x,y\in X$ and all $n\in \mathbf{N}$ implies $x=0$.
Finally, a nonzero vector $e\in X$\ is called a (\emph{strong order}) \emph{%
unit} if for each $x\in X$ there exists $\lambda \geq 0$ such that $-\lambda
e\leq x\leq \lambda e$. In what follows, $X$ will be assumed to be an
Archimedean Riesz space with unit $e$.

A vector subspace $N$\ of $X$ is a \emph{Riesz subspace} if $x\vee y\in N$
whenever $x,y\in N$. A vector subspace $N$ of $X$ is said to be an \emph{%
order ideal} if it is solid, i.e., $x\in N$ whenever there exists $y\in N$
such that $|x|\leq |y|$, where $|z|:=z\vee (-z)$. Since an order ideal $N$\
contains the absolute values of its elements and $x\vee y=\frac{1}{2}\left(
x+y+|x-y|\right) $ for all $x,y\in X$, an order ideal is automatically a
Riesz subspace. We denote its positive cone by $N_{+}:=\{x\in N:x\geq 0\}$.
Finally, an\ order ideal $N$ is said to be \emph{uniformly closed} if $x\in
N $ whenever there exist a sequence of vectors $(x_{n})$ in $N$, a vector $%
y>0$, and a decreasing sequence of positive reals $(\varepsilon _{n})$ such
that $\lim_{n}\varepsilon _{n}=0$ and $|x_{n}-x|\leq \varepsilon _{n}y$ for
all $n\in \mathbf{N}$. A detailed theory of Riesz spaces can be found in 
\cite{MR2011364}.

We endow $X$\ with the norm $\Vert \cdot \Vert :X\rightarrow \left[ 0,\infty
\right) $\ defined by%
\begin{equation}
\forall x\in X,\quad \Vert x\Vert :=\inf \{\lambda \geq 0:|x|\leq \lambda
e\}.  \label{eq:normdefinition}
\end{equation}%
Note that $\Vert \cdot \Vert $ is indeed a norm of\ $X$ which is also a
Riesz norm, that is,\ $\Vert x\Vert \leq \Vert y\Vert $ whenever $|x|\leq
|y| $. In particular, $\left( X,\Vert \cdot \Vert \right) $ is an\ $M$-space
and its norm dual, discussed below,\ is an $AL$-space (see \cite[pp. 93--98]%
{MR2011364}).

\begin{rmk}
\label{rmk:infimumminimum} Since $\left( X,\Vert \cdot \Vert \right) $ is a
normed Riesz space, the positive cone $X_{+}$ is closed (see, e.g., \cite[%
Theorem 2.21]{MR2011364}). This immediately yields that $\{\lambda \geq
0:|x|\leq \lambda e\}$ is closed. By the definition of $\Vert \cdot \Vert $,
the infimum in \eqref{eq:normdefinition} is thus attained. In particular, we
have that $|x|\leq \Vert x\Vert e$ for all $x$ in $X$. Finally, since $\Vert
\cdot \Vert $ is a Riesz norm, this latter fact implies that an order ideal
is uniformly closed if and only if it is norm closed.
\end{rmk}

Let $X^{\star }$ be the norm dual of $X$. We endow\ $X^{\star }$ and any of
its subsets with the weak$^{\star }$ topology and with the canonical
ordering induced by the positive cone of $X$, i.e., 
\begin{equation*}
X_{+}^{\star }:=\{\xi \in X^{\star }:\langle x,\xi \rangle \geq 0\text{ for
all }x\in X_{+}\},
\end{equation*}%
where $\langle x,\xi \rangle $ stands for the dual pairing. Finally, we
denote the positive unit sphere by%
\begin{equation*}
\Delta :=\{\xi \in X_{+}^{\star }:\langle e,\xi \rangle =1\}.
\end{equation*}%
By the Banach-Alaoglu Theorem, $\Delta $ is compact. Moreover, the set of
extreme points of $\Delta $\ is nonempty and compact as well. 

Following \cite{MR3403064}, we say that a functional $V: X\to \mathbf{R}$ is:
\begin{enumerate}[label={\rm (\roman{*})}]
\item \label{item:Vnormalized} \emph{normalized} if $V(\lambda e)=\lambda$ for all $\lambda \in \mathbf{R}$;
\item \label{item:Vmonotone} \emph{monotone} if $V(x)\le V(y)$ for all vectors $x\le y$ in $X$;
\item \label{item:Vunitadditive} \emph{unit-additive} if $V(x+\lambda e)=V(x)+V(\lambda e)$ for all $x \in X$ and all $\lambda\ge 0$; 
\item \label{item:Vunitmodular} \emph{unit-modular} if $V(x\vee \lambda e)+V(x\wedge \lambda e)=V(x)+V(\lambda e)$ for all $x \in X$ and all $\lambda\in \mathbf{R}$.
\end{enumerate}

\begin{rmk}
\label{rmk:continuity} If a functional $V:X\rightarrow \mathbf{R}$ satisfies
properties \ref{item:Vnormalized}--\ref{item:Vunitadditive}, then $V$ is
Lipschitz continuous of order $1$: in fact,%
\begin{equation*}
V(x)=V(y+(x-y))\leq V(y+\Vert x-y\Vert e)=V(y)+\Vert x-y\Vert
\end{equation*}%
for all $x,y\in X$ so that, by symmetry, $|V(x)-V(y)|\leq \Vert x-y\Vert $.
\end{rmk}

If $N\subseteq X$ is a proper uniformly closed order ideal, we say that a
functional $V:X\rightarrow \mathbf{R}$ is
\begin{enumerate}[label={\rm (\roman{*})}]
\setcounter{enumi}{4}
\item \label{item:Vinvariant} \emph{$N$-invariant} if $V(x)=V(y)$ whenever $x-y \in N$.
\end{enumerate} 

An important class of $N$-invariant functionals are those that are also
linear. In particular, we denote by $\Delta _{N}$ the subset of functionals
in the positive unit sphere $\Delta $ that annihilate $N$ and we denote by $%
\mathcal{E}_{N}$ the set of extreme points of $\Delta _{N}$,\ that is,%
\begin{equation}
\Delta _{N}:=\Delta \cap N^{\perp }\quad \text{ and }\quad \mathcal{E}_{N}:=%
\mathrm{ext}(\Delta _{N}),  \label{eq:defiEN}
\end{equation}%
where $N^{\perp }$ is the annihilator of $N$. Note that $\Delta _{N}$ is a
nonempty, convex, and compact subset of $X^{\star }$ (cf. Proposition \ref{prop:representationN}). By the Krein-Milman Theorem, also $\mathcal{E}_{N}$ is nonempty.

In what follows, we are going to show that a functional $V:X\rightarrow 
\mathbf{R}$\ which satisfies properties \ref{item:Vnormalized}--\ref{item:Vinvariant} has to be necessarily a Choquet average of the continuous
linear functionals in $\mathcal{E}_{N}$.

\begin{thm}
\label{thm:Archimedean} Let $X$ be an Archimedean Riesz space with unit $e$
and let $N$ be a proper uniformly closed order ideal of $X$. If $%
V:X\rightarrow \mathbf{R}$ is a functional which satisfies properties \ref%
{item:Vnormalized}--\ref{item:Vinvariant}, then there exists a normalized
capacity $\nu :\mathscr{B}(\mathcal{E}_{N})\rightarrow \mathbf{R}$ such that%
\begin{equation}
\forall x\in X,\quad V(x)=\int_{\mathcal{E}_{N}}\langle x,\xi \rangle \,%
\mathrm{d}\nu (\xi ).  \label{eq:claimchoquet}
\end{equation}
\end{thm}

In order to prove this result, we need four ancillary lemmas. From now on,
let $X$ and $N$ be as in the statement of Theorem \ref{thm:Archimedean}. We
denote by $\dot{X}$ the quotient space $X/N$. For each $x\in X$ denote by $%
\dot{x}:=\pi (x)$ the corresponding vector in $\dot{X}$, where $\pi
:X\rightarrow \dot{X}$ is the canonical projection. Recall that $\pi 
$ is an onto lattice homomorphism (see, e.g., \cite[Theorem 1.34]{MR2011364}%
).\ If a functional $V:X\rightarrow \mathbf{R}$ is $N$-invariant, then we
denote by $\dot{V}:\dot{X}\rightarrow \mathbf{R}$ the induced functional on
the quotient space, that is, $\dot{V}\left( \dot{x}\right) =V\left( x\right) 
$, where $\dot{x}=\pi (x)$ for some $x\in X$. Since $V$ is $N$-invariant, the
functional $\dot{V}$ is well defined. For,\ if $\dot{x}=\dot{y}$ for some $%
x,y\in X$, then $\pi (x-y)=0$, that is $x-y\in N$, and we can conclude that $%
\dot{V}(\dot{x})=V(x)=V(y)=\dot{V}(\dot{y})$.

\begin{lem}
\label{lem:Vwelldefined} $\dot{X}$ is an Archimedean Riesz space with unit $%
\dot{e}$.
\end{lem}

\begin{proof}
By \cite[Theorem 5.1]{MR217562}, $\dot{X}$ is an Archimedean Riesz space. We
are only left to show that $\dot{e}=\pi (e)$ is a unit of $\dot{X}$.
Consider $a\in \dot{X}$. It follows that $a=\pi (x)$ for some $x\in X$.
Since $e$ is a unit of $X$, there exists $\lambda \geq 0$\ such that $x\in
\lbrack -\lambda e,\lambda e]$. Since $\pi $ is a lattice homomorphism, $\pi 
$ is a positive linear operator, yielding that $a\in \lbrack -\lambda \dot{e}%
,\lambda \dot{e}]$. Finally,\ $\dot{e}$ is nonzero. By contradiction, assume
that $\pi (e)=\dot{e}=\dot{0}$. This would imply that $\lambda e\in N$ for
all $\lambda \geq 0$. Since$\ e$ is a unit and $N$ an order ideal, we would have
that for each $x\in X\ $there exists $\lambda \geq 0$ such that $\left\vert
x\right\vert \leq \lambda e$, yielding that $x\in N$ and, in particular, $%
X=N $, a contradiction with $N$\ being proper.
\end{proof}

\begin{lem}
\label{lem:propertiesV} 
Let $V: X\to \mathbf{R}$ be a $N$-invariant functional. 
If $V$ satisfies properties \ref{item:Vnormalized}--%
\ref{item:Vunitmodular}, so does\ $\dot{V}$.
\end{lem}

\begin{proof}
By Lemma \ref{lem:Vwelldefined}, $\dot{e}$ is a unit of $\dot{X}$. By
definition of $\dot{V}$, we have that%
\begin{equation*}
\forall \lambda \in \mathbf{R},\quad \dot{V}(\lambda \dot{e})=\dot{V}%
(\lambda \pi (e))=\dot{V}(\pi (\lambda e))=V(\lambda e)=\lambda ,
\end{equation*}%
proving that $\dot{V}$ is normalized. Fix $a,b\in \dot{X}$ such that $a\leq
b $. By \cite[p.17]{MR2011364},\ there exist $x,y\in X$ such that $\pi
\left( x\right) =a$, $\pi \left( y\right) =b$, and $x\leq y$. By the
monotonicity of $V$ and the definition of $\dot{V}$, we obtain that $\dot{V}%
(a)=V(x)\leq V(y)=\dot{V}(b)$, proving that\ $\dot{V}$ is monotone. Since $%
\pi $ is linear and a lattice homomorphism, the remaining properties are
straightforward to check. Fix $a\in \dot{X}$, so that $a=\pi \left( x\right) 
$\ for some $x\in X$. It follows that for each $\lambda \geq 0$%
\begin{equation*}
\begin{split}
\dot{V}(a+\lambda \dot{e})& =\dot{V}(\pi (x+\lambda e))=V(x+\lambda e) \\
& =V(x)+V(\lambda e)=\dot{V}(\pi \left( x\right) )+\dot{V}(\pi (\lambda e))=%
\dot{V}(a)+\dot{V}(\lambda \dot{e}).
\end{split}%
\end{equation*}%
Similarly, we have that for each $\lambda \in \mathbf{R}$%
\begin{equation*}
\begin{split}
\dot{V}(a\vee \lambda \dot{e})+\dot{V}(a\wedge \lambda \dot{e})& =\dot{V}%
(\pi (x)\vee \pi (\lambda e))+\dot{V}(\pi (x)\wedge \pi (\lambda e)) \\
& =\dot{V}(\pi (x\vee \lambda e))+\dot{V}(\pi (x\wedge \lambda e)) \\
& =V(x\vee \lambda e)+V(x\wedge \lambda e) \\
& =V(x)+V(\lambda e)=\dot{V}(a)+\dot{V}(\lambda \dot{e}).
\end{split}%
\end{equation*}%
We can conclude that $\dot{V}$ is unit-additive and unit-modular.
\end{proof}

\begin{lem}
\label{lem:Deltadot} If\ $V\in \Delta $ is $N$-invariant, then%
\begin{equation}
\dot{V}\in \dot{\Delta}:=\{\dot{\xi}\in \dot{X}^{\star }:\langle \dot{x},%
\dot{\xi}\rangle \geq 0\text{ for all }\dot{x}\in \dot{X}_{+}\text{ and }%
\langle \dot{e},\dot{\xi}\rangle =1\}.  \label{eq:deltadot}
\end{equation}
\end{lem}

\begin{proof}
Since $V\in \Delta $, it satisfies properties \ref{item:Vnormalized}--\ref%
{item:Vunitmodular}. Since $V$ is $N$-invariant, $\dot{V}$ is well defined
and clearly linear. By Lemma \ref{lem:propertiesV}, $\dot{V}$ is normalized
(thus, $\langle \dot{e},\dot{V}\rangle =\dot{V}(\dot{e})=1$), monotone
(thus, $\langle \dot{x},\dot{V}\rangle =\dot{V}(\dot{x})\geq \dot{V}(\dot{0}%
)=0$ for all $\dot{x}\in \dot{X}_{+}$), and unit-additive.\ By Remark \ref%
{rmk:continuity}, $\dot{V}$ is continuous,\ proving the statement.
\end{proof}

In what follows, the set $\dot{\Delta}$ defined in \eqref{eq:deltadot} will
be endowed with the weak$^{\star }$ topology $\sigma \left( \dot{\Delta},%
\dot{X}\right) $. Recall that $\Delta _{N}$ is the set of linear functionals 
$\xi \in \Delta $ that annihilate $N$, as defined in \eqref{eq:defiEN}.

\begin{lem}
\label{lem:affineembedding} The map $\Pi :\dot{\Delta}\rightarrow \Delta
_{N} $ defined by%
\begin{equation*}
\zeta \mapsto \zeta \circ \pi
\end{equation*}%
is an affine homeomorphism.
\end{lem}

\begin{proof}
First, we prove that $\Pi $ is well defined. Fix $\zeta \in \dot{\Delta}$
and define $\xi :=\Pi (\zeta )$, that is, $\xi (x)=\zeta (\pi (x))$ for all $%
x\in X$. Since $\zeta $ and $\pi $ are positive linear operators, it follows
that $\xi $ is linear and positive, that is,\ $\xi (x)=\zeta (\pi (x))\geq
\zeta (\dot{0})=0$ for all $x\in X_{+}$ as well as $\xi (e)=\zeta (\pi
(e))=\zeta (\dot{e})=1$. By Remark \ref{rmk:continuity}, $\xi $ is also
continuous, proving that\ $\xi \in \Delta $. Since $\pi (x)=\dot{0}\ $for
all $x\in N$, we have that $\xi (x)=\zeta (\pi (x))=\zeta (\dot{0})=0$ for
all $x\in N$, yielding that $\xi \in \Delta _{N}$.

We now show that $\Pi $ is bijective. To this aim, fix $\zeta _{1},\zeta
_{2}\in \dot{\Delta}$ and suppose that $\Pi (\zeta _{1})=\Pi (\zeta _{2})$,
that is,\ $\zeta _{1}(\dot{x})=\zeta _{2}(\dot{x})$ for all $x\in X$. Since $%
\pi $ is surjective, it follows that $\zeta _{1}=\zeta _{2}$, hence $\Pi $
is injective. Next, fix $\xi \in \Delta _{N}$ and $x,y\in X$ such that $%
x-y\in N$. Since $N$ is an order ideal, it follows that $|x-y|\in N$ and $%
\xi (|x-y|)=0$. Since $\xi \in X_{+}^{\star }$, this implies that $\xi
(x)\leq \xi (y+|x-y|)=\xi (y)$ and similarly $\xi (y)\leq \xi (x)$, proving
that $\xi $ is $N$-invariant. By Lemma \ref{lem:Deltadot}, we can conclude
that $\dot{\xi}\in \dot{\Delta}$. By definition of $\Pi $ and $\dot{\xi}$,
we have that $\Pi (\dot{\xi})(x)=\dot{\xi}(\pi (x))=\xi (x)$ for all $x\in X$%
, i.e., $\Pi (\dot{\xi})=\xi $, proving that $\Pi $ is surjective.

Finally, by construction, $\Pi $ is affine and continuous. Since $\dot{\Delta%
}$ and $\Delta _{N}$ are Hausdorff and compact, this implies that $\Pi $\ is
a\ homeomorphism.
\end{proof}

We can now prove Theorem \ref{thm:Archimedean}.

\begin{proof}[Proof of Theorem \protect\ref{thm:Archimedean}]
By\ Lemma \ref{lem:Vwelldefined}, $\dot{X}$ is an Archimedean Riesz space
with unit $\dot{e}$. By\ Lemma \ref{lem:propertiesV}, $\dot{V}$ satisfies
properties \ref{item:Vnormalized}--\ref{item:Vunitmodular}. We define%
\begin{equation*}
\mathscr{U}(\mathrm{ext}(\dot{\Delta})):=\{\{\dot{\xi}\in \mathrm{ext}(\dot{%
\Delta}):f(\dot{\xi})\geq t\}:f\in C(\mathrm{ext}(\dot{\Delta})),\,t\in 
\mathbf{R}\}
\end{equation*}%
the lattice of upper level sets generating\ the Baire $\sigma $-algebra of\ $%
\mathrm{ext}(\dot{\Delta})$. By \cite[Theorem 6]{MR3403064}, there exists a
(unique outer continuous) capacity $\mu :\mathscr{U}(\mathrm{ext}(\dot{\Delta%
}))\rightarrow \mathbf{R}$ such that%
\begin{equation*}
\forall \dot{x}\in \dot{X},\quad \dot{V}(\dot{x})=\int_{\mathrm{ext}(\dot{%
\Delta})}\dot{\xi}(\dot{x})\,\mathrm{d}\mu (\dot{\xi}).
\end{equation*}%
In addition, since $\dot{V}$ is normalized, we have $\mu (\mathrm{ext}(\dot{%
\Delta}))=1$. By Lemma \ref{lem:affineembedding}, $\Pi :\dot{\Delta}%
\rightarrow \Delta _{N}$\ is an affine homeomorphism. This implies that%
\begin{equation*}
\Pi \lbrack \,\mathrm{ext}(\dot{\Delta})]=\mathrm{ext}(\Delta _{N})=:%
\mathcal{E}_{N}.
\end{equation*}%
If we define%
\begin{equation*}
\mathscr{U}(\mathcal{E}_{N}):=\{\{\xi \in \mathcal{E}_{N}:f(\xi )\geq
t\}:f\in C(\mathcal{E}_{N}),\,t\in \mathbf{R}\},
\end{equation*}
then it is routine to check that%
\begin{equation*}
\mathscr{U}(\mathrm{ext}(\dot{\Delta}))=\{\Pi ^{-1}\left( E\right) :E\in %
\mathscr{U}(\mathcal{E}_{N})\}.
\end{equation*}

\noindent Define $\nu _{0}:\mathscr{U}(\mathcal{E}_{N})\rightarrow \mathbf{R}$ by $\nu
_{0}\left( E\right) :=\mu \left( \Pi ^{-1}\left( E\right) \right) $ for all $%
E\in \mathscr{U}(\mathcal{E}_{N})$. Clearly, $\nu _{0}$ is a normalized
capacity. 
%By \eqref{eq:Pi-upp-lev}, 
Note also that for each $x\in X$ and for
each $t\in \mathbf{R}$%
\begin{displaymath}
\begin{split}
\nu _{0}(\{\xi \in \mathcal{E}_{N}:\langle x,\xi \rangle \geq t\})&=\mu
\left( \{\dot{\xi}\in \mathrm{ext}(\dot{\Delta}):\langle x,\Pi (\dot{\xi}%
)\rangle \geq t\}\right) \\
&=\mu \left( \{\dot{\xi}\in \mathrm{ext}(\dot{\Delta}%
):\langle \dot{x},\dot{\xi}\rangle \geq t\}\right).
\end{split}
\end{displaymath}

By the definition of Choquet integral, this implies that%
\begin{equation}
\forall x\in X,\quad V(x)=\dot{V}(\dot{x})=\int_{\mathrm{ext}(\dot{\Delta})}%
\dot{\xi}(\dot{x})\,\mathrm{d}\mu (\dot{\xi})=\int_{\mathcal{E}_{N}}\xi (x)\,%
\mathrm{d}\nu _{0}(\xi ).  \label{eq:cha-var}
\end{equation}%
Since $\mathscr{U}(\mathcal{E}_{N})\subseteq \mathscr{B}(\mathcal{E}_{N})$,
the claim \eqref{eq:claimchoquet} follows by defining the normalized
capacity $\nu :\mathscr{B}(\mathcal{E}_{N})\rightarrow \mathbf{R}$ as, for
example,%
\begin{equation*}
\forall B\in \mathscr{B}(\mathcal{E}_{N}),\quad \nu (B):=\sup \{\nu
_{0}(A):A\in \mathscr{U}(\mathcal{E}_{N}),A\subseteq B\}.
\end{equation*}%
By definition of Choquet integral and since $\nu $ extends $\nu _{0}$, the
statement follows.
\end{proof}

We conclude the section with an example and a result. The example exhibits a class of proper uniformly closed order ideals. The result proves that these sets are the building blocks of any proper uniformly closed order ideal. 

\begin{rmk}
\label{rmk:exampleuniformlyclosedideals} For each $\xi \in \Delta $, the set%
\begin{equation}
N_{\xi }:=\{x\in X:\xi (|x|)=0\}  \label{eq:Nxi}
\end{equation}%
is a proper subset of $X$ (since $e\notin N_{\xi }$) with the following
properties:

\begin{enumerate}[label=(\alph*)]

\item $N_{\xi }$ is a vector subspace: if $x,y\in N_{\xi }$ and $\alpha ,\beta
\in \mathbf{R}$, then 
\begin{equation*}
0\leq \xi (|\alpha x+\beta y|)\leq |\alpha |\xi (|x|)+|\beta |\xi
(|y|)=0,\,\,\text{ hence }\alpha x+\beta y\in N_{\xi }\textup{;}
\end{equation*}

\item $N_{\xi }$ is solid (and, in particular, a Riesz subspace): if $x\in X$
and\ $y\in N_{\xi }$ are such that\ $|x|\leq |y|$, then 
\begin{equation*}
0\leq \xi (|x|)\leq \xi (|y|)=0,\,\,\text{ hence }x\in N_{\xi }\textup{;}
\end{equation*}

\item $N_{\xi }$ is uniformly closed: by Remark \ref{rmk:infimumminimum},
it\ is enough to check that $N_{\xi }$ is norm closed. Since $\Vert \cdot
\Vert $ is a Riesz norm and $||x|-|y||\leq |x-y|\ $for all $x,y\in X$, we
have that if $(x_{n})$ is a sequence in $N_{\xi }$ which converges in norm
to $x$, then $(\left\vert x_{n}\right\vert )$ converges in norm to $%
\left\vert x\right\vert $. Since $\xi $ is continuous, we can conclude that $%
\xi (|x|)=\lim_{n}\xi (|x_{n}|)=0$, that is,\ $x\in N_{\xi }$.
\end{enumerate}

Therefore $N_\xi$ is a proper uniformly closed order ideal of $X$.
\end{rmk}

%The next and final result shows that the above ideals $N_{\xi }$ are the
%building blocks of any proper ideal in $X$.

\begin{prop}
\label{prop:representationN} If\ $N$ is\ a proper uniformly closed order
ideal of an Archimedean Riesz space $X$ with unit $e$, then%
\begin{equation*}
N=\bigcap\nolimits_{\xi \in \Delta _{N}}N_{\xi },
\end{equation*}%
where $\Delta _{N}$ and $N_{\xi }$ are defined as in \eqref{eq:defiEN} and %
\eqref{eq:Nxi}, respectively.
\end{prop}

\begin{proof}
By Remark\ \ref{rmk:infimumminimum} and since $N$ is\ a proper uniformly
closed order ideal of $X$, then $N$ is norm closed and convex. 
Consider a point $z\in X\backslash N$. By the Hahn--Banach Theorem, there exists $\xi
\in X^{\star }$ such that $\xi \left( z\right) >\xi \left( x\right) $ for
all $x\in N$. We next show that $\xi \left( x\right) =0$ for all $x\in N$
and $\xi $ can be chosen to be in $\Delta $. Since $N$ is a vector space, we
have that $\xi \left( x\right) =0$ for all $x\in N$\ and, in particular, $%
\xi \left( z\right) \not=0$. Since $N$ is a Riesz subspace, we have that the
positive and negative part of each vector $x$ in $N$, that is $x^{+}$ and $%
x^{-}$, belong to $N$. By a Riesz--Kantorovich formula\ and since $X^{\star }
$ is an $AL$-space, we have that $\xi =\xi ^{+}-\xi ^{-}$, where $\xi
^{+},\xi ^{-}\in X_{+}^{\star }$\ and for each $x\geq 0$%
\begin{equation*}
\xi ^{+}\left( x\right) =\sup \left\{ \xi \left( y\right) :0\leq y\leq
x\right\} .
\end{equation*}%
Since $N$ is an order ideal, for each $x\in N$ we have that if $0\leq y\leq x$,
then $y\in N$. Since $\xi \left( N\right) =\left\{ 0\right\} $, this implies
that $\xi ^{+}\left( x\right) =0$ for all $x\in N_{+}$, proving that $\xi
^{+}\left( x^{+}\right) =\xi ^{+}\left( x^{-}\right) =0$ for all $x\in N$.
We can conclude that $\xi ^{+}\left( x\right) =\xi ^{+}\left(
x^{+}-x^{-}\right) =\xi ^{+}\left( x^{+}\right) -\xi ^{+}\left( x^{-}\right)
=0$ for all $x\in N$. Since $\xi =\xi ^{+}-\xi ^{-}$ and $\xi \left(
N\right) =\left\{ 0\right\} $, it follows that $\xi ^{-}\left( x\right) =0$ for all $x\in N$%
. Note that either $\xi ^{+}\left( z\right) \not=0$\ or $\xi ^{-}\left(
z\right) \not=0$. Otherwise, since $\xi =\xi ^{+}-\xi ^{-}$, we would have
that $0=\xi ^{+}\left( z\right) -\xi ^{-}\left( z\right) =\xi \left(
z\right) \not=0$, a contradiction. Assume $\xi ^{+}\left( z\right) \not=0$
(resp.\ $\xi ^{-}\left( z\right) \not=0$). In this case, we have that $\xi
^{+}\left( e\right) \not=0$ (resp. $\xi ^{-}\left( e\right) \not=0$).
Otherwise, since $\xi ^{+}\in X_{+}^{\star }$ (resp. $\xi ^{-}\in
X_{+}^{\star }$) and $e$ is a unit, we would have that $\xi ^{+}=0$ (resp. $%
\xi ^{-}=0$), a contradiction with $\xi ^{+}\left( z\right) \not=0$ (resp. $%
\xi ^{-}\left( z\right) \not=0$). In the first case, define\ $\bar{\xi}=\xi
^{+}/\xi ^{+}\left( e\right) $. In the second case, define $\bar{\xi}=\xi
^{-}/\xi ^{-}\left( e\right) $. Since $\xi ^{+},\xi ^{-}\in X_{+}^{\star }$,
we have that $\bar{\xi}\in \Delta $. Moreover, since $\xi ^{+}\left(
N\right) =\xi ^{-}\left( N\right) =\left\{ 0\right\} $, we can conclude that 
$\bar{\xi}\left( z\right) \not=0=\bar{\xi}\left( x\right) $ for all $x\in N$
and, in particular, $\bar{\xi}\in \Delta _{N}$.

By Remarks \ref{rmk:infimumminimum} and\ \ref%
{rmk:exampleuniformlyclosedideals} and since $N$ is an order ideal, $N_{\xi }$ is
a proper norm closed order ideal which contains $N$ for all $\xi \in \Delta
_{N}$, yielding that $M:=\bigcap\nolimits_{\xi \in \Delta _{N}}N_{\xi }$ has
the same properties. By contradiction, assume that there exists $z\in
M\backslash N$. By the initial part of the proof, there exists $\bar{\xi}\in
\Delta _{N}$ such that $\bar{\xi}\left( z\right) \not=0=\bar{\xi}\left(
x\right) $ for all $x\in N$. Since $z\in M$, we have that $z\in N_{\bar{\xi}%
} $. Since $\bar{\xi}\in \Delta \,$and $\bar{\xi}\left( \left\vert
z\right\vert \right) =0$, we have that $0\leq \bar{\xi}\left( z^{+}\right) ,%
\bar{\xi}\left( z^{-}\right) \leq \bar{\xi}\left( \left\vert z\right\vert
\right) =0$, yielding that\ $\bar{\xi}\left( z\right) =\bar{\xi}\left(
z^{+}\right) -\bar{\xi}\left( z^{-}\right) =0$, a contradiction.
\end{proof}

\section{Proof of Theorem \protect\ref{thm:mainN}}

\label{sec:mainproof}

Before we proceed with the proof of Theorem \ref{thm:mainN}, we need some
preliminary results related to ideal convergence. Given an ideal $\mathcal{I}
$ on $\mathbf{N}$, a real-valued\ sequence $x=(x_{n})$ is said to be $%
\mathcal{I}$-convergent to $\eta \in \mathbf{R}$, shortened as $\mathcal{I}%
\text{-}\lim x=\eta $,\ if $\{n\in \mathbf{N}:|x_{n}-\eta |\geq \varepsilon
\}\in \mathcal{I}$ for all $\varepsilon >0$. Accordingly, we define%
\begin{equation*}
c_{0}(\mathcal{I}):=\{x\in \mathbf{R}^{\mathbf{N}}:\mathcal{I}\text{-}\lim
x=0\},
\end{equation*}%
that is, $c_{0}(\mathcal{I})$ is the set of real-valued\ sequences which are 
$\mathcal{I}$-convergent to $0$. As usual, we denote by $\ell _{\infty }$\
the space of real-valued\ bounded sequences. Recall that $\ell _{\infty }$
is an Archimedean Riesz space with unit $e=(1,1,\ldots )$ (see, e.g., \cite[%
p. 538]{MR2378491}). Note that if $\mathcal{I}\neq \mathrm{Fin}$, then $%
c_{0}(\mathcal{I})$ is not contained in $\ell _{\infty }$ (choose, for
example, $(x_{n})$ such that $x_{n}=n$ if $n\in S$ and $x_{n}=0$ otherwise,
with $S\in \mathcal{I}\setminus \mathrm{Fin}$). The set of bounded $\mathcal{%
I}$-convergent sequences, $c_{0}(\mathcal{I})\cap \ell _{\infty }$, has been
studied in several works (see, e.g., \cite{MR2735533, MR1734462, MR3863065,  
MR3883309, MR2181783, MR3836186}).

In what follows, given a set $A\subseteq \mathbf{N}$, let $x=\bm{1}_{A}$ be
the characteristic function of $A$, so that $x_{n}=1$ if $n\in A$ and $%
x_{n}=0$ otherwise (hence, $e=\bm{1}_{\mathbf{N}}$).

\begin{lem}
\label{lem:uniformlyclosedideal} If $\mathcal{I}$ is an ideal on $\mathbf{N}$%
, then $c_{0}(\mathcal{I})\cap \ell _{\infty }$ is a proper uniformly closed
order ideal of $\ell _{\infty }$.
\end{lem}

\begin{proof}
It is easy to see that $c_{0}(\mathcal{I})$ is a vector subspace of $\mathbf{%
R}^{\mathbf{N}}$ and so is $\ell _{\infty }$. Since $\mathcal{I}\text{-}\lim
e=1$, we obtain that $c_{0}(\mathcal{I})\cap \ell _{\infty }$ is a proper
vector subspace of $\ell _{\infty }$. Consider $y\in c_{0}(\mathcal{I})\cap
\ell _{\infty }$ and $x\in \ell _{\infty}$ such that $|x|\leq |y|$. By the
definition of $\mathcal{I}$-convergence, it follows that $\mathcal{I}$-$\lim
x=0$, proving that $c_{0}(\mathcal{I})\cap \ell _{\infty }$ is an order ideal.
Finally, it is routine to check that $c_{0}(\mathcal{I})\cap \ell _{\infty }$
is norm closed. By Remark \ref{rmk:infimumminimum}, $c_{0}(\mathcal{I})\cap \ell_{\infty }$ is uniformly closed, proving the statement.
\end{proof}

Denote by $ba$ the space of signed finitely additive measures\ $\mu :\mathcal{P}(\mathbf{N})\rightarrow \mathbf{R}$ with finite total variation,
that is, such that 
\begin{equation*}
\sup \left\{ \sum\nolimits_{i=1}^{n}|\mu (A_{i})|:\{A_{1},\ldots ,A_{n}\}%
\text{ is a partition of }\mathbf{N}\right\} <\infty
\end{equation*}%
(see, e.g., \cite[Section 10.10]{MR2378491}). Recall that $\ell _{\infty
}^{\star }$ can be identified with $ba$ (see, e.g., \cite[Theorem 14.4]%
{MR2378491}) via the lattice isomorphism $T:\ell _{\infty }^{\star
}\rightarrow ba$ defined by%
\begin{equation*}
\forall \xi \in \ell _{\infty }^{\star },\forall A\subseteq \mathbf{N},\quad
T(\xi )(A):=\xi (\bm{1}_{A}).
\end{equation*}%
We endow both $ba$ and the norm dual $\ell _{\infty }^{\star }$ with the weak%
$^{\star }$ topology. Note that $T$ is continuous and its inverse is given by%
\begin{equation}
\forall \mu \in ba,\forall x\in \ell _{\infty },\quad T^{-1}(\mu )(x)=\int_{%
\mathbf{N}}x\,\mathrm{d}\mu .  \label{eq:Tinverse}
\end{equation}

\begin{lem}
\label{lem:MI} %Let $X=\ell _{\infty }$. 
If $\mathcal{I}$ is an ideal on $\mathbf{N}$ and $N=c_{0}(\mathcal{I})\cap \ell _{\infty }$, then%
$$
T\textup{[}\Delta_N\textup{]}=\mathrm{M}(\mathcal{I}),
$$
where $\mathrm{M}(\mathcal{I})$ is the set of finitely additive probability
measures\ $\mu $ such that $\mu (A)=0$ for all $A\in \mathcal{I}$.
\end{lem}

\begin{proof}
Set $X=\ell_\infty$. 
By definition of $T$, it is immediate to see that $T(\xi )$ is a finitely
additive probability measure\ for all $\xi \in \Delta $. Consider $\xi \in
\Delta _{N}=\Delta \cap N^{\perp }$. Since $\xi \in N^{\perp }$, if $A\in 
\mathcal{I}$, then $\bm{1}_{A}\in N$, yielding that $T(\xi )(A)=\xi (\bm{1}%
_{A})=0$ for all $A\in \mathcal{I}$ and proving that $T\textup{[}\Delta_N\textup{]}\subseteq \mathrm{M}(\mathcal{I})$. 

Conversely, fix $\mu \in \mathrm{M}(\mathcal{I})$. Given Equation %
\eqref{eq:Tinverse}, define $\xi :=T^{-1}(\mu )\in \ell _{\infty }^{\star }$%
. It is straightforward to check that $\xi \in \Delta $. Consider $x\in
N_{+} $. Since $\mathcal{I}$-$\lim x=0$, $x\geq 0$, and $\mu \in \mathrm{M}(%
\mathcal{I})$, we have that $\{n\in \mathbf{N}:x_{n}\geq t\}\in \mathcal{I}$
and $\mu (\{n\in \mathbf{N}:x_{n}\geq t\})=0$ for all $t>0$, proving that%
\begin{equation}
\forall x\in N_{+},\quad \xi (x)=\int_{\mathbf{N}}x\,\mathrm{d}\mu
=\int_{0}^{\infty }\mu (\{n\in \mathbf{N}:x_{n}\geq t\})\,\mathrm{d}t=0.
\label{eq:nul}
\end{equation}%
Since $N$ is a Riesz subspace, we have that the positive and negative part
of each vector $x$ in $N$, that is $x^{+}$ and $x^{-}$, belong to $N$. By %
\eqref{eq:nul}, we can conclude that $\xi (x)=\xi (x^{+})-\xi (x^{-})=0$ for
all $x\in N$, proving that $\xi \in N^{\perp }$ and, in particular, that $%
\xi \in \Delta \cap N^{\perp }=\Delta _{N}$. This implies that $T(\xi )=\mu $, proving that $\mathrm{M}(\mathcal{I}) \subseteq T\textup{[}\Delta_N\textup{]}$. 
\end{proof}

As a side result, we obtain a representation of $c_0(\mathcal{I}) \cap
\ell_\infty$.

\begin{prop}
If\ $\mathcal{I}$ is an ideal on $\mathbf{N}$, then 
\begin{equation*}
c_{0}(\mathcal{I})\cap \ell _{\infty }=\left\{ x\in \ell _{\infty }:\int_{%
\mathbf{N}}|x|\,\mathrm{d}\mu =0\text{ for all }\mu \in \mathrm{M}(\mathcal{I%
})\right\} .
\end{equation*}
\end{prop}

\begin{proof}
Set $X=\ell _{\infty }$ and $N=c_{0}(\mathcal{I})\cap \ell _{\infty }$. By
Lemma \ref{lem:uniformlyclosedideal}, $N$ is a proper uniformly closed order
ideal of $X$. By Proposition \ref{prop:representationN}, Lemma \ref{lem:MI},
and Equation \eqref{eq:Tinverse}, we obtain that%
\begin{eqnarray*}
N &=&\bigcap\nolimits_{\xi \in \Delta _{N}}\{x\in X:\xi
(|x|)=0\}=\bigcap\nolimits_{\mu \in T[\Delta _{N}]}\{x\in X:T^{-1}(\mu
)(|x|)=0\} \\
&=&\bigcap\nolimits_{\mu \in \mathrm{M}(\mathcal{I})}\{x\in X:\int_{\mathbf{N%
}}|x|\,\mathrm{d}\mu =0\},
\end{eqnarray*}%
proving the statement.
\end{proof}

We are finally ready to prove Theorem \ref{thm:mainN}.

\begin{proof}[Proof of Theorem \protect\ref{thm:mainN}]
\ref{item:a1} $\implies $ \ref{item:a2} Let $X$ be the Archimedean Riesz
space $\ell _{\infty }$ with unit $e=\bm{1}_{\bm{N}}$ and $S\ $the Riesz
subspace of sequences which take\ finitely many values. Define $N:=c_{0}(%
\mathcal{I})\cap \ell _{\infty }$. By Lemma \ref{lem:uniformlyclosedideal}, $%
N$ is a proper uniformly closed order ideal. Define the functional $%
V:X\rightarrow \mathbf{R}$ by%
\begin{equation*}
\forall x\in X,\quad V(x):=\int_{\mathbf{N}}x\,\mathrm{d}\nu .
\end{equation*}%
Note that $V$ satisfies properties \ref{item:Vnormalized}--\ref%
{item:Vunitmodular} (see, e.g., \cite[Proposition 4.11, Theorem 4.3, and Lemma 4.6]{marinacci2004}).  
We next show that it is $N$-invariant.

\begin{claim}
\label{claim1} $V(x+z)=V(x)$ whenever $x \in S_+$ and $z \in N_+$.
\end{claim}

\begin{proof}
Consider $x\in S_{+}$ and $z\in N_{+}$. Define for each $\varepsilon >0$%
\begin{equation*}
A_{\varepsilon }:=\{n\in \mathbf{N}:x_{n}+z_{n}\geq \varepsilon \}\setminus
\{n\in \mathbf{N}:x_{n}\geq \varepsilon \}.
\end{equation*}%
Fix $\varepsilon >0$. If $K_{\varepsilon }:=\{x_{n}:x_{n}<\varepsilon \}$ is
empty, then $A_{\varepsilon }=\emptyset $ and $A_{\varepsilon }\in \mathcal{I%
}$. Since $x\in S$ and $\mathcal{I}$-$\lim z=0$, if $K_{\varepsilon }$ is a
nonempty set, then $K_{\varepsilon }$ is finite\ and%
\begin{equation*}
A_{\varepsilon }=\{n\in \mathbf{N}:x_{n}\in K_{\varepsilon }\text{ and }%
z_{n}\geq \varepsilon -x_{n}\}\subseteq \{n\in \mathbf{N}:z_{n}\geq
\varepsilon -\max K_{\varepsilon }\}\in \mathcal{I},
\end{equation*}%
proving that $A_{\varepsilon }\in \mathcal{I}$. Since $x,z\geq 0$, $\{n\in 
\mathbf{N}:x_{n}\geq \varepsilon \}\setminus \{n\in \mathbf{N}%
:x_{n}+z_{n}\geq \varepsilon \}\in \mathcal{I}$ for all $\varepsilon >0$
(being empty). By\ the definition of Choquet integral and since $\nu $ is $%
\mathcal{I}$-invariant and $x,z\geq 0$, this implies that $V(x+z)=V(x)$.
\end{proof}

\begin{claim}
\label{claim2} $V(x+z)\le V(x)$ whenever $x \in S$ and $z \in N$.
\end{claim}

\begin{proof}
Consider\ $x\in S_{+}$ and $z\in N$. Since $N$ is an order ideal, $|z|\in
N_{+}$. By Claim \ref{claim1} and since $V$ is monotone, it follows that $%
V(x+z)\leq V(x+|z|)=V(x)$. Next, consider $x\in S$. Since $e$ is a unit,
there exists $\lambda >0$ such that $x+\lambda e\geq 0$ and, clearly, $%
x+\lambda e\in S_{+}$. Since $V$ is unit-additive and normalized, it follows
that%
\begin{equation*}
V(x+z)+\lambda =V(x+\lambda e+z)\leq V(x+\lambda e)=V(x)+\lambda ,
\end{equation*}%
proving the claim.
\end{proof}

\begin{claim}
\label{claim3} $V$ is $N$-invariant.
\end{claim}

\begin{proof}
Fix $x,y\in X$ such that $x-y\in N$. Since $S$ is dense in $X$, there exist
two sequences $(x^{k})$ and $(y^{k})$ in $S$ which are norm convergent to $x$
and $y$, respectively. We obtain by Claim \ref{claim2} that%
\begin{equation*}
\forall k\in \mathbf{N},\quad V(y^{k}+x-y)\leq V(y^{k})\,\,\text{ and }%
\,\,V(x^{k}+y-x)\leq V(x^{k}).
\end{equation*}%
By Remark \ref{rmk:continuity}, $V$ is continuous. By passing to the limit,
this implies that $V(x)\leq V(y)$ and $V(y)\leq V(x)$.
\end{proof}

By Theorem \ref{thm:Archimedean} and since $V$ satisfies properties \ref%
{item:Vnormalized}--\ref{item:Vinvariant}, there exists a normalized
capacity $\psi :\mathscr{B}(\mathcal{E}_{N})\rightarrow \mathbf{R}$ such that%
\begin{equation}
\forall x\in X,\quad V(x)=\int_{\mathcal{E}_{N}}\langle x,\xi \rangle \,%
\mathrm{d}\psi (\xi ).  \label{eq:almostconclusion}
\end{equation}

\noindent Define%
\begin{equation*}
\mathrm{F}(\mathcal{I}):=\{\mu _{\mathcal{F}}:\mathcal{F}\in \mathrm{Ult}(%
\mathcal{I})\}.
\end{equation*}

\begin{claim}
\label{claim:extremeMI} $T(\mathcal{E}_N)=\mathrm{F}(\mathcal{I})$.
\end{claim}

\begin{proof}
By Lemma \ref{lem:MI}, the map from $\Delta _{N}$ to $\mathrm{M}(\mathcal{I}%
) $, defined by $\xi \mapsto T(\xi )$, is an affine bijection. This implies
that $T(\mathcal{E}_{N})=\mathrm{ext}(\mathrm{M}(\mathcal{I}))$.\ The proof
that $\mathrm{ext}(\mathrm{M}(\mathcal{I}))=\mathrm{F}(\mathcal{I})$ goes
verbatim as in the case $\mathcal{I}=\left\{ \emptyset \right\} $ (see,
e.g., \cite[p. 544]{MR2378491}).
\end{proof}

Define the capacity $\kappa :\mathscr{B}(\mathrm{F}(\mathcal{I}))\rightarrow 
\mathbf{R}$ by $\kappa (A):=\psi (T^{-1}(A))$. By the same arguments used in
proving \eqref{eq:cha-var}\ and using \eqref{eq:Tinverse}, %
\eqref{eq:almostconclusion}, and Claim \ref{claim:extremeMI}, we obtain that 
\begin{equation*}
\forall x\in X,\quad V(x)
%=\int_{T(\mathcal{E}_{N})}T^{-1}(\mu )(x)\,\mathrm{d%}\rho (\mu )
=\int_{\mathrm{F}(\mathcal{I})}\,\left( \int_{\mathbf{N}}x\,%
\mathrm{d}\mu \right) \,\mathrm{d}\kappa (\mu ),
\end{equation*}%
Lastly, by a similar reasoning, since the map $\mathrm{Ult}(\mathcal{I})\rightarrow \mathrm{F}(%
\mathcal{I})$ defined by $\mathcal{F}\mapsto \mu _{\mathcal{F}}$ is a homeomorphism, we get
\begin{equation*}
\forall x\in X,\quad V(x)=\int_{\mathrm{Ult}(\mathcal{I})}\,\left( \int_{%
\mathbf{N}}x\,\mathrm{d}\mu _{\mathcal{F}}\right) \,\mathrm{d}\rho (\mathcal{%
F}),
\end{equation*}%
which concludes the proof.

\medskip

\ref{item:a2} $\implies$ \ref{item:a3} Choose $x=\bm{1}_A$ with $A\subseteq 
\mathbf{N}$.

\medskip

\ref{item:a3} $\implies $ \ref{item:a1} 
It follows by the fact that 
$\mu _{\mathcal{F}}$ is $\mathcal{I}$-invariant 
for all $\mathcal{F} \in \mathrm{Ult}(\mathcal{I})$. 
\end{proof}

\section*{Acknowledgments}

Simone Cerreia-Vioglio and Massimo Marinacci gratefully acknowledge the financial support of ERC (grants SDDM-TEA and INDIMACRO, respectively). Paolo Leonetti and Fabio Maccheroni acknowledge the financial support of PRIN (grant 2017CY2NCA). 

%%%%%%%%%%%%%%%%%%%%%%%%%%%%%%%%%%%%%%%%%%%%%%%%%%%%%%%%%%%%%%%%%%%%%%%
%\nocite{*}
\bibliographystyle{amsplain}
\bibliography{choquet}
%\begin{thebibliography}{99}
%\end{thebibliography}

\end{document}